\documentclass[a4paper, reqno]{amsart}
\usepackage{amsthm}
\usepackage{amsfonts}
\usepackage{amsmath}
\usepackage{mathrsfs}
\usepackage[pdfauthor={Oliver Butterley},
		pdftitle={A Note on Operator Semigroups Associated to Chaotic Flows},
		pdfsubject={Semigroups and Chaotic Flows},
		pdfkeywords={Chaotic flow, Operator semigroups},
		colorlinks=true,
		citecolor=black,
		linkcolor=black]{hyperref}
\usepackage{bbm}

\numberwithin{equation}{section}

\newtheorem{ass}{Assumption}
\newtheorem{asss}{Assumption}

\newtheorem{thm}{Theorem}[section]
\newtheorem{thmm}{Theorem}
\newtheorem{lem}[thm]{Lemma}

\theoremstyle{definition}

\newtheorem{rem}[thm]{Remark}

\DeclareMathAlphabet{\mathsc}{OT1}{cmr}{m}{sc}

\newcommand\cB{{\mathscr B}}
\newcommand\cC{{\mathscr C}}

\newcommand\cL{{\mathcal L}}

\newcommand\cM{{\mathscr M}}
\newcommand\cP{{ P}}

\newcommand\cN{{\mathcal N}}

\newcommand\Cone{{C_{1}}}
\newcommand\Ctwo{{C_{2}}}
\newcommand\Cdolgo{C_{3}}
\newcommand\Cten{C_{4}}
\newcommand\Cjim{{C_{5}}}
\newcommand\Cthree{{C_{6}}}
\newcommand\Cfour{{C_{7}}}
\newcommand\Cfive{{C_{8}}}
\newcommand\Cmid{{C_{9}}}
\newcommand\Couter{{C_{10}}}
\newcommand\Cthirteen{C_{11}}

\newcommand\Celeven{s}
\newcommand\Ctwelve{r}

\newcommand\Cjune{{C_{\ell}}}

\newcommand\Cjohn{{\gamma_{0}}}

\newcommand\bC{{\mathbb C}}

\newcommand\bN{{\mathbb N}}

\newcommand\bR{{\mathbb R}}

\newcommand\As{{\mathcal A}}
\newcommand\Bs{{\mathcal B}}
\newcommand\Ds{{\mathcal D}}

\newcommand{\T}[1]{T_{#1}} 
\newcommand{\Id}{\mathbf{id}}   

\newcommand{\norm}[1]{\left\lVert{#1}\right\rVert}
\newcommand{\abs}[1]{\left\lvert{#1}\right\rvert}
\newcommand{\R}[1]{R({#1})}

\newcommand{\flo}[1]{\Phi^{{#1}}} 

\begin{document}
\author{ Oliver Butterley}
\title[Operator Semigroups Associated to Chaotic Flows]{A Note on Operator Semigroups Associated to Chaotic Flows}
\date{10\textsuperscript{th} March 2016}
\email{oliver.butterley@univie.ac.at}
\address{Oliver Butterley\\
 Fakult\"at f\"ur Mathematik\\
Universit\"at Wien\\
Oskar-Morgenstern-Platz 1\\ 1090 Wien, Austria.}
\thanks{ It is a pleasure to thank Viviane Baladi and Carlangelo Liverani for helpful discussions. Research supported by the Austrian Science Fund, Lise Meitner position M1583. 
 Additional 
 support by the Stiftung Aktion \"Osterreich Ungarn (A\"OU), Projekt Nr. 87\"ou6.}

\begin{abstract}
The transfer operator associated to a flow (continuous time dynamical system) is a one-parameter operator semigroup.
We consider the operator-valued Laplace transform of this one-parameter semigroup. Estimates on the Laplace transform have been used in various settings in order to show the rate at which the flow mixes. 
Here we consider the case of exponential mixing and the case of rapid mixing (super polynomial). We develop the  operator theory framework amenable to this setting and show that the same estimates may be used to produce results, in terms of the operators, which go beyond the results for the rate of mixing. 
\end{abstract}

\maketitle

\thispagestyle{empty}

\section{Introduction}\label{sec:intro}
Flows are important dynamical systems, arguable the origin of much of the research in the area of dynamical systems. It has proved significantly more difficult to study strong statistical properties of flows compared to corresponding questions for discrete time systems. Of particular importance is proving the rate of mixing of a given system or family of systems. Substantial initial progress was made by studying the Laplace transform of the correlation function~\cite{Po,D}. A certain estimate (the oscillatory cancelation estimate pioneered by Dolgopyat~\cite{D}) can then be translated into an exponential mixing estimate for the flow.
 These ideas were  developed by Liverani to the closely related question of studying the resolvent operator of the infinitesimal generator of the semigroup of transfer operators~\cite{Li1}. An identical argument is used by Baladi and Liverani~\cite{Baladi:2011uq}, and by Giuletti, Pollicott and Liverani~\cite{Giulietti:fk}.
 We further develop these ideas, extending the idea of considering the operator-valued Laplace transform of the Transfer operator~\cite{oli1202} and show that one may squeeze a little more information from this line of thinking.
The case for exponentially mixing flows and rapid mixing flows are presented side-by-side in the same language and so are easily comparable. 

The improved operator-theoretic result is of interest in several ways. Firstly that beyond the rate of mixing there are many other statistical properties with can often be deduced from the spectral results~\cite[\S9]{Keller:1989} and cannot be deduced directly from the rate of mixing.
Another important use of the functional analysis is for studying how statistical properties behave under perturbations of the dynamical system~\cite{KL}. Such perturbations could be deterministic or random. Moreover the same ideas (as interpreted in~\cite{keller2005sgo}) can  be used for the physically important question of understanding coupled dynamical systems.

The improvements are as follows. 
The calculation involved is completely streamlined. This makes it clear that the constant obtained in the decay rate (in terms of degree of differentiability of the observable) cannot be improved without additional ideas.
We avoid the need for Liverani's ``silly preliminary fact''~\cite[Lemma~2.14]{Li1}.
Additionally we are able to obtain a spectral decomposition \eqref{eq:result} of the transfer operator in a sense similar, although weaker, than the results of Tsujii~\cite{Ts,tsujii2008qct,Tsujii:2011uq}. This means we obtain a precise description of the mixing and moreover it is to be expected that further information concerning other statistical properties can be obtained from this operator-theoretic  representation. The result is in a form especially amenable to the ideas of~\cite{KL} regarding the use of operator perturbation theory in order to understand various questions in dynamical systems.

Note that in this article we do not prove the required estimates for any particular dynamical systems with respect to any particular Banach space. Rather we  isolate the abstract argument and make some improvements to this. 
It is an important question and a subject of ongoing research to investigate the rate of mixing (and other fine statistical properties) for a broad spectrum of flows. 
The method we are discussing (i.e., functional analysis applied to dynamical systems) requires as a first step the choice or design of a Banach space on which the one-parameter family of transfer operators acts ``nicely''. 
Moreover, at this point in time, to answer such questions for flows, no one knows a method which does not involve  functional-analytic ideas to some extent. 
Designing appropriate Banach spaces and proving such estimates for systems of interest (including many physically relevant systems with discontinuities and singularities) remains an important subject of ongoing research (see, for example, \cite{Demers:2011fk, oli1203}).
In many cases the appropriate choice of Banach space is far from obvious. 
In this note we are able to reduce the assumptions that such a  Banach space must satisfy in order to be useful and consequently simplify the search for and construction of the dynamically relevant Banach spaces. In particular we avoid the requirement that the one-parameter semigroup is strongly continuous.

In view of potential numerical applications throughout the argument we will keep track of all the relevant constants. In Section~\ref{sec:results} we present the results in two theorems, one concerning the exponentially mixing case and the other concerning the rapid mixing case. In Section~\ref{sec:apps} we give details of systems where the required assumptions have already been shown to be satisfied. We hope these assumptions will soon be shown to be satisfied in many more settings. 
Section~\ref{sec:proofs1} and Section~\ref{sec:proofs2} are devoted to the proofs of the results.

\section{Results}\label{sec:results}
Let $(\Bs, \norm{\cdot}_{\Bs})$ and   $(\As, \norm{\cdot}_{\As})$    be Banach spaces such that $\As\supset \Bs$ and   $\norm{\cdot}_{\As} \leq  \norm{\cdot}_{\Bs}$.\footnote{In actual fact one needs only the Banach space $(\Bs, \norm{\cdot}_{\Bs})$ equipped with an auxiliary, weaker norm $ \norm{\cdot}_{\As}$. However in this case one can always define $\As$ to be the completion of $\Bs$ with respect to $ \norm{\cdot}_{\As}$ and so without loss of generality with give the assumptions as above. } 
Let $\cB(\Bs,\Bs)$ denote the Banach space of bounded linear operators $T:\Bs \to \Bs$ equipped with the   standard operator norm which we denote $\norm{T}_{\Bs}$.
We consider a measurable operator-valued function $T:[0,\infty) \to \cB(\Bs,\Bs)$ denoted by $t \mapsto \T{t}$  such that 
\[
\T{0} = \Id, \quad \quad
\T{s}\circ \T{t} = \T{s+t} \quad \text{for all $t,s \geq 0$},
\]
and that
 $\norm{\T{t}}_{\Bs } \leq \Cone$ for some  $\Cone>0$.
In other words $\T{t}:\Bs \to \Bs$ is a bounded one-parameter  semigroup.\footnote{The boundedness requirement  is essentially superfluous since if a one-parameter semigroup  satisfies a bound of the form $\norm{\T{t}}_{\Bs } \leq C e^{\gamma t}$ then we may simply consider the operator $\tilde{\T{t}} := e^{-\gamma t}\T{t}$ and proceed as before.  }  
 We  define a weaker operator norm
\begin{equation}\label{eq:defopnorm}
\norm{T}_{\Bs \to \As} := \sup\{\norm{T \mu}_{\As}: \mu\in \Bs, \norm{\mu}_{\Bs}\leq 1\}.
\end{equation}
It would be unrealistic in the intended applications to hope that the semigroup is norm continuous, often the semigroup is not even strongly continuous.\footnote{The one-parameter semigroup $\T{t}$ is said to be \emph{strongly continuous} if $T : [0,\infty) \times \Bs \to \Bs$ is jointly continuous. 
Given the semigroup structure it is only required to check the continuity at~$0$:
It is known~\cite[Theorem~6.2.1]{Davies:2007qy} that  $\T{t}$ is strongly continuous iff $\lim_{t\to 0}\T{t} \mu = \mu$ for all $\mu\in \Bs$.  There are examples~\cite[Example~6.1.10]{Davies:2007qy} such that $T : (0,\infty) \times \Bs \to \Bs$ is jointly continuous but $T : [0,\infty) \times \Bs \to \Bs$ is not jointly continuous.}  We merely require the following, substantially weaker, continuity condition. 
\begin{ass}[Weak-Lipschitz]\label{ass:Lip}
There exists $\Ctwo>0$ such that
\[
\frac{1}{t}\norm{\T{t}-\Id}_{\Bs \to \As} \leq \Ctwo\quad \text{for all $t\geq 0$}.
\]
\end{ass}
\noindent
See Section~\ref{sec:apps} for discussion of this assumption and how it is natural in the intended applications.
For all $z\in \bC$, $\Re(z)>0$ let $\R{z}\in \cB(\Bs,\Bs)$ be defined by the  Bochner integral
\begin{equation}\label{eq:defRz}
\R{z} := \int_{0}^{\infty}e^{-zt}\T{t} \ dt.
\end{equation}
Since the semigroup is bounded we know that $\norm{\R{z}}_{\Bs} \leq \Cone \Re(z)^{-1}$ for all $\Re(z)>0$ but we need a bit more information concerning $\R{z}$.
\begin{ass}\label{ass:LY}
There exists $\lambda>0$ such that 
the essential spectral radius of $\R{z}:\Bs \to\Bs$ is not greater than $(\Re(z)+\lambda)^{-1}$  for all $\Re(z)>0$. 
\end{ass}
\noindent
In all cases we will assume the both Assumption~\ref{ass:Lip} and Assumption~\ref{ass:LY} hold. In addition we will assume that one of the two following assumptions holds. 
The first is an oscillatory cancellation type estimate of the form used by Dolgopyat in the study of Anosov flows~\cite{D}.
\begin{asss}[Exponential]\label{ass:dolgo}
There exists $\beta, \alpha, \Cdolgo>0$ and $\gamma \in(0, 1/\ln(1+\lambda/\alpha))$ such that,
 for all $\Re(z)= \alpha$, $\abs{\Im(z)}\geq \beta$,
\[
\norm{\R{z}^{\tilde n}}_{\Bs} \leq  \Cdolgo (\Re{(z)}+\lambda)^{-\tilde n}, \quad \quad \text{where $\tilde n =\lceil \gamma\ln \abs{\Im(z)}\rceil$}.
\]
\end{asss} 
\noindent
An alternative and far weaker assumption is the following estimate of the form used by Dolgopyat in the study of the prevalence of rapid mixing among Axiom~A flows~\cite{dolgopyat1998prm}.
\begin{asss}[Rapid]\label{ass:rapid}
There exists  $\beta,\Cten, \Celeven, \Ctwelve >0$ such that
$\R{z}$ admits a holomorphic extension to the set $\{z \in \bC: \abs{\Im(z)} \geq \beta,   \Re(z) \geq - \abs{\Im(z)}^{-\Ctwelve}\}$ and on this set 
\[
\norm{\R{z}}_{\Bs} \leq \Cten \abs{\Im(z)}^{\Celeven}.
\]
\end{asss}
\noindent
That Assumption~\ref{ass:dolgo} is stronger than Assumption~\ref{ass:rapid} can be seen from the calculations in Section~\ref{sec:proofs1}.

For the moment suppose that  $\T{t} : \Bs \to\Bs$ is a strongly continuous  one-parameter semigroup. The \emph{generator} of the  semigroup is the linear operator 
defined by
\[ 
Z\mu := \lim_{t\to 0} \frac{1}{t}  \left(\T{t} \mu -\mu   \right)
\]
the domain of $Z$, which we denote $\operatorname{Dom}(Z)$, being the set of $\mu\in \Bs$ for which the limit exists. 
There is no reason to expect $Z$ to be a bounded operator. 
It is known~\cite[Lemma~6.1.15]{Davies:2007qy} that $ \operatorname{Dom}(Z)$ is complete with respect to the norm
\[
\norm{\mu}_{Z} := \norm{Z\mu}_{\Bs}+ \norm{\mu}_{\Bs},
\]
and consequently $Z$ is a closed operator~\cite[Problem~6.1.1]{Davies:2007qy}. 
According to~\cite[Lemma~6.1.11]{Davies:2007qy}  we know that $\operatorname{Dom}(Z)$ is $\norm{\cdot}_{\Bs}$-dense in $\Bs$.

In the case when $\T{t} : \Bs \to\Bs$ is not strongly continuous it will be convenient to have a subset of $\Bs$ on which $\T{t}$ is known to be well behaved. Let, just as in~\cite{oli1202},
\[
\Ds_{0}:= \left\{{ \textstyle \int_{0}^{s} } \T{t}\mu \ dt : \mu \in \Bs, s>0  \right\}.
\]
Let $\Ds \subseteq \Bs$ denote the completion, with respect to $\norm{\cdot}_{\Bs}$, of $\Ds_{0}$.
Assumption~\ref{ass:Lip} implies that $\Ds$ is $\norm{\cdot}_{\As}$-dense in $\Bs$. 
(If it were known that $\T{t} : \Bs \to\Bs$ is strongly continuous then $\Ds$ is $\norm{\cdot}_{\Bs}$-dense in $\Bs$.)
It is easy to see that $\T{t} \Ds \subseteq \Ds$, and a simple estimate~\cite[Lemma~2.8]{oli1202} implies that
$\norm{ \T{t}\mu - \mu}_{\Bs} \to 0$ as $t\to 0$ for all $\mu \in \Ds$.
Consequently $\T{t} : \Ds \to \Ds$ is a strongly-continuous one-parameter semigroup.\footnote{Of course $\Ds$ should be understood to mean the Banach space $(\Ds, \norm{\cdot}_{\Bs})$.}
In this case we define the generator $Z$ with $\operatorname{Dom}(Z)$ as above but for the semigroup $\T{t} : \Ds \to \Ds$.
In this case  we know that $\Ds_{0} \subseteq \operatorname{Dom}(Z)$ and consequently $\operatorname{Dom}(Z)$ is $\norm{\cdot}_{\As}$-dense in $\Bs$.
From this point forward when we refer to  $\operatorname{Dom}(Z)$ this  should be understood to imply the Banach space $(\operatorname{Dom}(Z), \norm{\cdot}_{Z})$,
defined as above, depending on whether the semigroup is strongly continuous or not.

 The first main result of this paper is the following theorem.
 \begin{thmm}\label{thm:main}
 Suppose that $\T{t}:\Bs \to \Bs$ is a bounded one-parameter semigroup satisfying Assumptions~\ref{ass:Lip}, \ref{ass:LY}, and \ref{ass:dolgo}. 
 Then there exists a finite set 
 \[
 \{z_{j}\}_{j=1}^{N}\subset \{z\in \bC: -\lambda<  \Re({z})\leq 0, \abs{\Im(z)}\leq \beta\},
 \]
   a set of finite rank projectors ${\{\Pi_{j}\}}_{j=1}^{N}$,
  a set of nilpotents  ${\{\cN_{j}\}}_{j=1}^{N}$ 
  and an operator-valued function $t\mapsto \cP_{t} \in \cB(\Bs, \Bs)$ where 
  $\Pi_{j}\cP_{t} = \cP_{t}\Pi_{j} = 0$,
  $\Pi_{j}\Pi_{k} = \delta_{jk} \Pi_{j}$,
  $\Pi_{j}\cN_{j} = \cN_{j}\Pi_{j} = \cN_{j}$
  such that 
 \begin{equation}\label{eq:result}
 \T{t} = \cP_{t} + \sum_{j=1}^{N} e^{tz_{j}}e^{t \cN_{j}} \Pi_{j}  \quad \text{for all $t\geq 0$}.
 \end{equation}
 Moreover for all $\ell < \lambda$ there exists $\Cjune>0$ such that,  for all $\mu\in \operatorname{Dom}(Z)$, $t\geq 0$, 
 \begin{equation}\label{eq:estimate}
 \norm{\cP_{t}\mu}_{ \As} \leq \Cjune e^{-\ell t}\norm{Z\mu}_{\Bs}.
 \end{equation}
 \end{thmm}
\noindent
 The proof of the theorem is the content of Section~\ref{sec:proofs1}.

 \begin{rem}
The theorem is only useful if the set $\operatorname{Dom}(Z)$ is sufficiently large. 
As discussed immediately prior to the theorem, $\operatorname{Dom}(Z)$ is  $\norm{\cdot}_{\As}$-dense in $\Bs$.
However if $\T{t}$ were a  strongly-continuous one-parameter semigroup, then  $\operatorname{Dom}(Z)$ is $\norm{\cdot}_{\Bs}$-dense in $\Bs$.
 \end{rem}
 \begin{rem}
 \label{rem:perturb}
 Since parts of the above result are limited to $\operatorname{Dom}(Z)$ one might suppose that it would have been convenient to work with the reduced operator semigroup $\T{t} : \Ds \to \Ds$ instead of $\T{t} : \Bs \to \Bs$ from the very beginning. 
 Sometimes this is convenient but, as illustrated in~\cite{oli1202}, this can easily cause problems when studying flows and their perturbations, particular systems with discontinuities. 
 The crucial problem being that the Banach space $\Ds$ depends on the flow, as does $\operatorname{Dom}(Z)$. 
 Note that in the part of the above result which concerns the peripheral spectrum does not depend on $\Ds$ or $\operatorname{Dom}(Z)$ allowing for the possibility of studying perturbation from an operator theory point of view. 
 \end{rem}

\begin{rem}\label{rem:zd}
For the purpose of this remark denote by $Z_{\Ds}$ the generator associated to $\T{t}: \Ds \to \Ds$ (this is the operator denoted by $Z$ throughtout the rest of the document). 
Similarly denote by  $Z_{\Bs}$ the generator associated to $\T{t}: \Bs \to \Bs$.
It holds that
\[
 \Ds_{0} \subseteq \operatorname{Dom}(Z_{\Ds})
  \subseteq \operatorname{Dom}(Z_{\Bs})
  \subseteq \Ds.
\]
The first inclusion was discussed in the above paragraph and the second is obvious. Here we will prove the final inclusion.
Let $\mu \in \operatorname{Dom}(Z_{\Bs})$.
For all $s>0$ let $\nu_{s} := s^{-1} \int_{0}^{s} \T{t}\mu \ dt \in \Ds_{0}$. 
Note that $\nu_{s} - \mu = s^{-1} \int_{0}^{s}( \T{t}\mu - \mu ) \ dt$.
Since $\lim_{t\to 0} \frac{1}{t}(\T{t}\mu - \mu) = Z_{\Bs}\mu \in \Bs$ then 
$\norm{ \nu_{s} - \mu }_{\Bs} \to 0$ as $s\to 0$ and so $\mu \in \Ds$.
\end{rem}

 \begin{rem}
With the current ideas we cannot hope for a strengthening of the theorem whereby $\norm{\cP_{t}}_{ \Bs} \leq Ce^{-\ell t}$.  This is a subtlety of one-parameter semigroups as demonstrated by Zabczyk's example\footnote{There exists a one-parameter group $\T{t}$ acting on a Hilbert space such that the spectrum of $Z$ is contained in $i\bR$ but $e^{\abs{t}}$ is in the spectrum of $\T{t}$ for all $t\in \bR$. This means that the inclusion proved in~\cite[Theorem 8.2.7]{Davies:2007qy} cannot be improved to an equality.}~\cite[Theorem 8.2.9]{Davies:2007qy}.
This is a problem that was overcome in the work of Tsujii~\cite{Ts,tsujii2008qct,Tsujii:2011uq} but results are limited to systems which are rather regular and it is not clear if such a strategy is possible in general.
 \end{rem}
\begin{rem}
If the one-parameter semigroup was actually a one-parameter semigroup of 
 operators associated to an ergodic flow, as in the intended applications, then  
one  can typically show that mixing is equivalent to $\{z_{j}\}_{j=1}^{N} \cap \{\Re(z) = 0\} = \{0\}$ (see, for example~\cite{BL, BL2}). 
\end{rem}
\begin{rem}
Most often Assumption~\ref{ass:LY} is proven by the combination of a compact embedding $\Bs \hookrightarrow \As$ 
and an estimate of the form $\norm{\R{z}^{n}\mu}_{\Bs} \leq C (\Re(z) + \lambda)^{-n} \norm{\mu}_{\Bs} + C (1+ \Im(z))\norm{\mu}_{\As}$. 
Such information is sufficient to deduce the estimate of the essential spectral radius by following Hennion's argument~\cite{He} based on the formula by Nussbaum~\cite{Nussbaum} (see for example~\cite{Li1}). 
In this case Assumption~\ref{ass:dolgo} can be weakened: It is then sufficient to prove the estimate of Assumption~\ref{ass:dolgo} 
in the weaker norm $\norm{\cdot}_{\As}$ rather than in the original norm $\norm{\cdot}_{\Bs}$ and only for $\mu\in \Bs$ for which $\norm{\mu}_{\Bs}$ is sufficiently small in comparison to  $\norm{\mu}_{\As}$.
\end{rem}

In order to state the result which corresponds to rapid mixing we must have higher order control on the regularity in the flow direction. For any $q\in \bN$ define the norm
\[
 \norm{\mu}_{Z^{q}} := \sum_{0 \leq n \leq q} \norm{Z^{n} \mu}_{\Bs},
\]
for all $\mu \in \operatorname{Dom}(Z^{q})$.
As before $Z$ is understood to be the generator of the strongly-continuous one-parameter semigroup $\T{t}: \Ds \to \Ds$.
The second main result of this paper is the following theorem.
 \begin{thmm}\label{thm:main2}
 Suppose that $\T{t}:\Bs \to \Bs$ is a bounded one-parameter semigroup satisfying Assumptions~\ref{ass:Lip}, \ref{ass:LY}, and \ref{ass:rapid}. 
 Then there exists a finite set 
 \[
 \{z_{j}\}_{j=1}^{N}\subset \{z\in \bC: -\lambda<  \Re({z})\leq 0, \abs{\Im(z)}\leq \beta\},
 \]
  a set of finite rank projectors $\{\Pi_{{j}}\}_{j=1}^{N} \subset \cB(\Bs,\Bs)$ and an operator-valued function $t\mapsto \cP_{t} \in \cB(\Bs, \Bs)$ such that 
 \begin{equation}\label{eq:result2}
  \T{t} = \cP_{t} + \sum_{j=1}^{N} e^{tz_{j}}e^{t \cN_{j}} \Pi_{j}  \quad \text{for all $t\geq 0$}.
 \end{equation}
 Moreover for all $p\in \bN$ there exists $q\in \bN$, $C_{p}>0$ such that,  for all $\mu\in \operatorname{Dom}(Z^{q})$, $t\geq 0$, 
 \begin{equation}\label{eq:estimate2}
 \norm{\cP_{t}\mu}_{ \As} \leq C_{p} t^{-p} \norm{\mu}_{Z^{q}}. 
 \end{equation}
 \end{thmm}
\noindent
 The proof of the theorem is the content of Section~\ref{sec:proofs2}.
\begin{rem}
Note that the required regularity $q = q(p)$ depends on the desired decay rate $p$ and must be taken larger when $p$ increases. The exact connection of the two can be seen in the calculation at the end of Section~\ref{sec:proofs2}.
When considering rates of mixing of a flow the above requirement of $\mu\in \operatorname{Dom}(Z^{q})$ becomes the unfortunate requirement of the observables being ``rather smooth'' in the flow direction.  
\end{rem}

\begin{rem}
Dolgopyat's original formulation~\cite{dolgopyat1998prm} of rapid mixing considered $\cC^{\infty}$ functions as observables.  
See \cite[Definition~2.2]{melbourne2007rdc} for a formulation closer to the above statement.
Note that being \emph{sufficiently regular} in the flow direction is crucial for this result.
 However it is of some help that the notion of regularity is entirely dependent on the choice of $\norm{\cdot}_{\Bs}$ and $\norm{\cdot}_{\As}$.
\end{rem}

\begin{rem}
Usually Assumption~\ref{ass:rapid} is proven by showing the non-existence of \emph{approximate eigenvalues}~\cite{dolgopyat1998prm,dolgopyat2000prm,melbourne2007rdc,FMT,MR2472164}.
\end{rem}

\section{Applications}\label{sec:apps}
Assumption~\ref{ass:LY} and Assumption~\ref{ass:dolgo} have been shown for contact Anosov flows ~\cite{Li1}   (in the reference the two spaces $\Bs$ and $\As$ are denoted $\mathcal{B}(\mathcal{M},\bC)$ and $\mathcal{B}_{w}(\mathcal{M},\bC)$ respectively).  In order to show that Assumption~\ref{ass:Lip}  holds it is convenient to modify the stronger of the two norms by adding a term which controls (in supremum) the derivative in the flow direction.  As a result assumption~\ref{ass:Lip} is simple to prove in this setting once one notices that
$\int_{0}^{t} V \eta \circ \flo{s} \ ds = \eta \circ \flo{t} - \eta$ for all $t\geq 0$ where $V$ is the vector field associated to the flow $\flo{t}:\cM\to \cM$. Let $\cL_{t}$ be the associated transfer operator.
This means that 
\[
\int_{\cM} (\cL_{t}h - h) \cdot \eta  \ dm
= \int_{\cM} h \cdot ( \eta\circ\flo{t} - \eta) \ dm
= \int_{0}^{t} \int_{\cM} \cL_{s}h \cdot V\eta \ dm \ ds.
\]
This immediately implies the weak Lipschitz control required by Assumption~\ref{ass:Lip}.
Similarly these assumptions  have been shown to be satisfied in several other settings ~\cite{Baladi:2011uq,Giulietti:fk}.

The observant reader will have noticed that the modification of the norm as described above has the unfortunate side effect that the Banach space is then dependent on the dynamics and therefore unsuitable to studying perturbations as outlined in Remark~\ref{rem:perturb}.
The improved norms~\cite{BL,BL2} for Anosov flows are immediately suitable in terms of satisfying Assumption~\ref{ass:Lip} and Assumption~\ref{ass:LY}. Unfortunately, at present, it is not known if Assumption~\ref{ass:dolgo} is satisfied with respect to these norms and for Anosov flows without contact structure all indications suggest that some new idea is required.

Assumptions~\ref{ass:Lip}, \ref{ass:LY}, and Assumption~\ref{ass:rapid} have been shown for a prevalent set of Axiom~A flows in~\cite{dolgopyat1998prm}. However in the reference everything is described in the \emph{twisted transfer operator} language for suspension flows. To pass from that viewpoint to the present  language note that the calculation~(see for example \cite{Po} or \cite[Lemma~7.17]{avila2005emt}) used to relate the Laplace transform of the correlation to a sum of twisted transfer operators may equally well be used for the Laplace transform of the transfer operator of the flow for the suspension flow.

\section{The Exponentially Mixing Case}\label{sec:proofs1}
Throughout we suppose that Assumptions~\ref{ass:Lip}, \ref{ass:LY}, and \ref{ass:dolgo} are satisfied. 
First we recall a fact which appeared in~\cite{oli1202}.
Note that the proof is done using the integral definition of $\R{z}$ (and not by associating it to a resolvent of some operator) using that Fubini also holds for Bochner integrals~\cite[Theorem~1.1.9]{arendt2011vector}.
\begin{lem}[{\cite[Lemma~2.2]{oli1202}}]
\label{lem:resolventequation}
For all $\Re(z)>0$, $\Re(\zeta)>0$ then, on $\cB(\Bs,\Bs)$, holds
\[
(z-\zeta) \R{\zeta} \R{z} = \R{\zeta} - \R{z}.
\]
\end{lem}

We already know that the operator valued function $z\mapsto \R{z} \in \cB(\Bs, \Bs)$   is holomorphic on the set $\{z\in \bC: \Re(z)>0\}$  \cite[Theorem~1.5.1]{arendt2011vector}.
We now take advantage of Assumption~\ref{ass:dolgo} for the following result.
\begin{lem}  
\label{lem:extension}
The operator valued function $z\mapsto \R{z} \in \cB(\Bs, \Bs)$  admits an extension which is  meromorphic   on the set $\{z\in \bC: \Re(z)>-\lambda\}$ and holomorphic on the set $\{z\in \bC: \Re(z)>-\lambda, \abs{\Im(z)}\geq \beta\}  $.
\end{lem}
\begin{proof}
Consider $z\in \bC$, $\Re(z)>0$ and $\eta\in \bC$, $\abs{\eta} > \Re(z)^{-1}$.
By Lemma~\ref{lem:resolventequation}
$ \eta^{-1} \R{z+\eta^{-1}} \R{z} = \R{z+\eta^{-1}} - \R{z}$
  since in particular  $\eta \neq 0$ and $\Re(z-\frac{1}{\eta})>0$. 
  Consequently
\begin{equation}\label{eq:pres}
\R{z+\tfrac{1}{\eta}} = \eta \R{z} (\eta \Id - \R{z})^{-1}.
\end{equation}
We know that $(\eta \Id - \R{z})$ is invertible since the spectral radius of $\R{z}$ is not greater than $\Re(z)^{-1}$.
Consequently \eqref{eq:pres} defines the extension of $\R{z}$ into the left half of the imaginary plane.
 By  Assumption~\ref{ass:LY} the operator valued function $\eta \mapsto (\eta \Id - \R{z})^{-1}$ is meromorphic on the set $\{ \abs{\eta} > (\Re(z) +\lambda)^{-1} \}$.
 By Assumption~\ref{ass:dolgo} we know that the spectral radius of $\R{z}$ is not greater than $(\Re(z) +\lambda)^{-1}$ when $\Re(z)>-\lambda$ and $ \abs{\Im(z)}\geq \beta$. This means that in this case the operator valued function $\eta \mapsto (\eta \Id - \R{z})^{-1}$ is holomorphic on this set.
\end{proof}

\begin{proof}[Proof of the first part of Theorem~\ref{thm:main}]
An immediate consequence of  Lemma~\ref{lem:extension} is that the function $z\mapsto \R{z} \in \cB(\Bs, \Bs)$  has no more than a finite number of poles on the set $\{z\in \bC: \Re(z)>-\lambda\}$. We let $\{z_{j}\}_{j=0}^{N} \subset \bC$ denote this finite set of poles. 
For each $z_{j}$ let
\[
\Pi_{j}:= \frac{1}{2\pi i} \int_{\Gamma_{j}} \R{z} \ dz
\]
where $\Gamma_{j}$ is a positively-orientated small circle enclosing $z_{j}$ but excluding all other singularities of $\R{z}$. As is well known for spectral projectors, the resolvent equation, which was proven in Lemma~\ref{lem:resolventequation}, implies that the definition is  independent on the choice of $\Gamma_{j}$ subject to the above conditions.
We now, for all $t\geq0$, define $\cP_{t}: \Bs \to \Bs$ by
 \[
 \cP_{t} := \T{t} - \sum_{j=1}^{N} e^{tz_{j}} \Pi_{{j}}.
 \]
To complete  the proof of the theorem  it remains to give the appropriate estimates on $\cP_{t}$. This is the substantial part of the present argument and will be postponed until the end of the section.
\end{proof}

The following key step is  an application of the inverse of the Laplace-Stieltjes transform of an operator valued function~\cite[Theorem~2.3.4]{arendt2011vector} to the present situation.
\begin{lem}
\label{lem:inverse}
Suppose $t\geq 0$, $a >0$. Then, in $\cB(\Bs, \As)$, we have that
\[
  \T{t}   = \lim_{k\to \infty} \frac{1}{2\pi i} \int_{-k}^{k} e^{(a+ib)t}\R{a+ib} \ db.
\]
\end{lem}
\noindent
Details for passing from the formulation in the reference~\cite{arendt2011vector} and the present setting can be found in~{\cite[Theorem~1]{oli1202}}
(using crucially Assumption~\ref{ass:Lip}).

The whole idea of the present argument is to obtain better information on $\R{z}$ and then use the  formula given by the above lemma whilst shifting the contour.

\begin{lem}\label{lem:extension2}
Suppose that $\ell \in (0,\lambda)$. For all $b\in \bR$, $\abs{b}\geq \beta$, on $\cB(\Bs,\Bs)$
\[
\R{-\ell+ib} = \R{\alpha +ib} \left( \sum_{n=0}^{\infty} (\alpha+\ell)^{n} \R{\alpha+ib}^{n}   \right).
\] 
Moreover  there exists $\Cjim>0$ such that for all $\abs{b}\geq \beta$
\[
\norm{  \vphantom{\sum}\smash{\sum_{n=0}^{\infty}} (\alpha+\ell)^{n} \R{\alpha+ib}^{n} }_{\Bs} \leq 
\Cjim \abs{b}^{\Cjohn}
\]
where $\Cjohn: = \gamma \ln (1+ \ell \alpha^{-1}) \in(0,1)$.
\end{lem}
\begin{proof}
Since the extension of $\R{z}$ was defined in Lemma~\ref{lem:extension} by the resolvent equation we have
\[
\begin{split}
\R{-\ell+ib} &=  \R{\alpha +ib} \left[ \Id - (\alpha+\ell) \R{\alpha +ib}\right]^{-1}\\
&=    \R{\alpha +ib} \left( \sum_{n=0}^{\infty} (\alpha+\ell)^{n} \R{\alpha +ib}^{n}\right).
\end{split}
\]
It is convenient to split the sum as
\[
\sum_{n=0}^{\infty} (\alpha+\ell)^{n} \R{\alpha +ib}^{n}
= \sum_{k=0}^{\infty}  (\alpha+\ell)^{k\tilde n(b)} \R{\alpha +ib}^{k \tilde n} \sum_{m=0}^{\tilde n(b)-1}  (\alpha+\ell)^{m} \R{\alpha +ib}^{m},
\]
where $\tilde n(b) =\lceil \gamma\ln \abs{b}\rceil$.
We use the estimate $\norm{\R{\alpha +ib}}_{\Bs} \leq \Cone \alpha^{-1}$ and the estimate $\norm{\R{\alpha +ib}^{\tilde n(b)}}_{\Bs} \leq \Cdolgo({\alpha }+\lambda)^{-\tilde n(b)}$ of Assumption~\ref{ass:dolgo}.
The norm of the first sum decreases  as $\abs{b}$ increases and so we have
\[
 \sum_{k=0}^{\infty}  (\alpha+\ell)^{k\tilde n(b)} \norm{\smash{\R{\alpha +ib}^{k \tilde n(b)}}}_{\Bs}
 \leq \Cthree,
\]
where $ \Cthree:= \Cdolgo [{1- (  \frac{\alpha+\ell}{\alpha+\lambda} )^{ \gamma\ln {\beta}}}]^{-1}$.
The norm of  the second sum is  increasing as $\abs{b}$ increases. We have
\[\begin{split}
\sum_{m=0}^{\tilde n(b)-1}  {(\alpha+\ell)}^{m} \norm{\R{\alpha+ib}^{m}}_{\Bs}
&\leq
\Cone  \sum_{m=0}^{\tilde n(b)-1} \left(\tfrac{\alpha+\ell}{ \alpha} \right)^{m} \\
&\leq \Cone \alpha \ell^{-1} \abs{b}^{\Cjohn},
\end{split}
\]
 recalling that  $\Cjohn = \gamma \ln (1+ \ell \alpha^{-1})$.
We let $\Cjim:= \Cone\Cthree \alpha \ell^{-1}$.
The above  two estimates complete the proof of the lemma.
\end{proof}

\begin{lem}\label{lem:cancel}
There exists $\Cfour>0$ such that for all $\abs{b}\geq \beta$
\[
\norm{\R{\alpha + ib}}_{\Bs \to \As} \leq \Cfour \abs{b}^{-1}.
\]
\end{lem}
\begin{proof}
This lemma is a consequence of Assumption~\ref{ass:Lip}.
Fix $b\in \bR$.
For all $n\in \bN$ let $t_{n}:= 2\pi n \abs{b}^{-1}$ and hence
\[
\begin{split}
\R{\alpha + ib}
&= \sum_{n=0}^{\infty} \int_{t_{n}}^{t_{n+1}} e^{-(\alpha+ib)t} \T{t} \ dt\\
&= \sum_{n=0}^{\infty}e^{-\alpha t_{n}}  \int_{t_{n}}^{t_{n+1}} e^{-ibt}\left( e^{-\alpha(t-t_{n})} \T{t}  - \T{t_{n}} \right) \ dt,
\end{split}
\]
since $\int_{t_{n}}^{t_{n+1}} e^{-ibt}  \ dt = 0$. 
We have that
$\abs{e^{-\alpha(t-t_{n})} -1} 
 \leq \alpha(t-t_{n}) \leq  2\pi \alpha \abs{b}^{-1}$.
Using Assumption~\ref{ass:Lip} we have that 
 $\norm{ \T{t}  - \T{t_{n}}  }_{\Bs \to \As} \leq (t- t_{n}) \Ctwo \Cone \leq 2\pi \Ctwo \Cone \abs{b}^{-1}$ for all $t\in(t_{n},t_{n+1})$.
 This means that
 \[
 \norm{ e^{-\alpha(t-t_{n})} \T{t}  - \T{t_{n}}  }_{\Bs \to \As}
 \leq 2\pi \Cone  \left( \alpha + \Ctwo  \right) \abs{b}^{-1}.
 \]
On the other hand
\[
\begin{split}
 \sum_{n=0}^{\infty}e^{-\alpha t_{n}}  \int_{t_{n}}^{t_{n+1}}  dt 
 &=\sum_{n=0}^{\infty} \frac{2\pi}{\abs{b}}e^{-\alpha 2 \pi n \abs{b}^{-1}} \\
 &= \frac{ 2\pi \abs{b}^{-1}  } { 1- e^{-\alpha2\pi \abs{b}^{-1}}  }
 \leq  \frac{ 2\pi {\beta}^{-1}  } { 1- e^{-\alpha2\pi {\beta}^{-1}}  }=:\Cfive.
\end{split}
\]
We have shown that $\norm{\R{\alpha + ib}}_{\Bs \to \As} \leq \Cfour  \abs{b}^{-1}$ where $\Cfour:= 2\pi \Cone  \Cfive  \left( \alpha + \Ctwo  \right)$.
\end{proof}

\begin{lem}\label{lem:Z}
For all  $z \in \bC$ in the holomorphic domain of $\R{z}$ and $z\neq 0$, on $\cB(\operatorname{Dom}(Z), \Bs)$
\[
\R{z} - \frac{1}{z}\Id = \frac{1}{z} \R{z}Z.
\]
\end{lem}
\begin{proof}
The claimed result concerns only $\operatorname{Dom}(Z)$ and so it suffices to consider $\T{t} : \Ds \to \Ds$, which, as discussed in the paragraph proceeding Theorem~\ref{thm:main}, is a strong-continuous one-parameter semigroup (see also Remark~\ref{rem:zd}).  
Consequently, by standard~\cite[Theorem 8.2.1]{Davies:2007qy} semigroup theory $\R{z} = (z\Id - Z)^{-1}$. This means that $\R{z} (z\Id - Z) = \Id = z\R{z} - \R{z}Z$.
\end{proof}

\begin{proof}[Proof of the second part of Theorem~\ref{thm:main}]
Let $\Re(a)>0$ and let $\ell <\lambda$ such that  $\Re(z_{j})>-\ell$ for all $j$.
By Lemma~\ref{lem:inverse} and shifting the contour of integration, remembering that $\R{z}$ has a pole at each $\{z_{j}\}_{j=1}^{N}$, we have, on $\cB(\Bs,\As)$, for all $t\geq0$
\[
\begin{split}
  \T{t}   &= \lim_{k\to \infty} \frac{1}{2\pi i} \int_{-k}^{k} e^{(a+ib)t}\R{a+ib} \ db\\
  &= \lim_{k\to \infty} \frac{1}{2\pi i} \int_{-k}^{k} e^{(-\ell+ib)t} \R{-\ell+ib}\ db
  + \sum_{j=1}^{N}  \frac{1}{2\pi i} \int_{\Gamma_{j}} e^{zt} \R{z} \ dz.
  \end{split}
\]
Since $R(z)$ is a pseudo-resolvent we represent $R(z)$ by the Laurent expansion in terms of projectors and nilpotents (see \cite[III-(6.35)]{Kato1966} and note that the quasi-nilpotents are actually nilpotents since they are finite rank in this case).
Consequently
 \[
 \begin{aligned}
     \frac{1}{2\pi i} \int_{\Gamma_{j}} e^{zt} R(z) \ dz
 &=    \frac{1}{2\pi i} \int_{\Gamma_{j}} e^{zt} \Big[  \frac{\Pi_{j}}{z-z_{j}} + \sum_{n=1}^{\infty} \frac{\cN_{j}^{n}}{(z-z_{j})^{n+1}}  \Big] \ dz\\
 &= e^{z_j t} \Big(  \Pi_{j} + \sum_{n=1}^{\infty} \frac{t^n}{n!} \cN_{j}^{n}  \Big)
 = e^{tz_{j}}e^{t \cN_{j}} \Pi_{j}.
 \end{aligned}
 \]
This means that, for all $t\geq0$,
\begin{equation}\label{eq:Pt}
\cP_{t} =  \lim_{k\to \infty} \frac{e^{-\ell t}}{2\pi i} \int_{-k}^{k} e^{ib t} \R{-\ell+ib}\ db.
\end{equation}
Since
$\int_{-\infty}^{\infty} \frac{e^{(\ell+ib)t} }{\ell+ib} \ db =0$
 we have 
\[
  \cP_{{t}}   = \lim_{k\to \infty} \frac{e^{-\ell t}}{2\pi i} \int_{-k}^{k} e^{ib t}\left(  \R{-\ell+ib} - \frac{\Id}{-\ell+ib}\right) \ db. 
\]
By Lemma~\ref{lem:Z} we have that, on $\As$, for every $\mu\in\operatorname{Dom}(Z) $ 
\[
 \cP_{{t}}\mu   = \lim_{k\to \infty} \frac{e^{-\ell t}  }{2\pi i} \int_{-k}^{k} e^{ib t}  \frac{\R{-\ell+ib}Z\mu}{-\ell+ib}  \ db.
\]
We must estimate $\norm{  \cP_{{t}}\mu }_{\As}$.
Note that $\norm{\smash{  \R{-\ell+ib}Z\mu}   }_{ \As}   \leq \norm{ \smash{ \R{-\ell+ib}}}_{\Bs\to \As}\norm{Z\mu}_{\Bs}$.
Let
\[
\Cmid :=   \frac{1}{2\pi } \int_{-\beta}^{\beta}   \frac{\norm{\smash{\R{-\ell+ib}}}_{\Bs \to \As}}{\abs{\smash{-\ell+ib}}}  \ db.
\]
Since the contour $\{z\in \bC, \Re(z) = -\ell, \abs{ \Im{z}} \leq \beta\}$ was chosen to avoid all the singularities of $\R{z}$ we have that $\Cmid <\infty$.
By Lemma~\ref{lem:cancel} and Lemma~\ref{lem:extension2}
$\norm{\smash{\R{-\ell+ib}}}_{\Bs \to \As}\leq  \Cfour \Cjim \abs{b}^{-(1-\Cjohn)}$ for all $\abs{b}\geq \beta$. Since $(1-\Cjohn) \in (0,1)$
\[
\Couter :=  \frac{1}{2\pi } \int_{\beta}^{\infty} \abs{b}^{-(2-\Cjohn)}   \ db <\infty.
\]
We have  shown that
$
\norm{  \cP_{{t}}\mu }_{\As} \leq  \Cjune e^{-\ell t}  \norm{Z\mu}_{\Bs}
$
where $\Cjune :=  \left( \Cmid + 2 \Couter \Cfour \Cjim   \right)$.
\end{proof}

\section{The Rapid Mixing Case}\label{sec:proofs2}

Throughout we suppose that Assumptions~\ref{ass:Lip}, \ref{ass:LY}, and \ref{ass:rapid} are satisfied. 
The argument follows closely~\cite{dolgopyat1998prm} but instead of using a Taylor expansion we take advantage of the generator $Z$.

\begin{proof}[Proof of the first part of Theorem~\ref{thm:main2}]
As before we use Lemma~\ref{lem:inverse} to write that
\begin{equation}\label{eq:last}
  \T{t}   = \lim_{k\to \infty} \frac{1}{2\pi i} \int_{-k}^{k} e^{(a+ib)t}\R{a+ib} \ db.
\end{equation}
Identically to the proof of the first part of Theorem~\ref{thm:main} we deal with the part of the integral from $-\beta$ to $\beta$ by selecting a finite set of projectors $\{ \Pi_{{j}} \}_{j}$ corresponding to the poles $\{ z_{j}\}_{j}$ of $\R{z}$ in the region $\{z \in \bC: \Re(z) > -\ell, \abs{\Im(z)} \leq \beta\}$.
We define (as before)
 \[
 \cP_{t} := \T{t} - \sum_{j=1}^{N} e^{tz_{j}} \Pi_{{j}}.
 \]
It now remains to estimate $ \norm{  \cP_{{t}}\mu }_{\As} $ in terms of $\norm{ \mu }_{Z^{q}}$ (for some $q \in \bN$) crucially using Assumption~\ref{ass:rapid}. 
It is convenient to shift the contour of integration to $\{ i b - \min(\epsilon, \abs{b}^{-\Ctwelve}), b \in \bR\}$ where $\epsilon \in (0,\ell)$ is chosen such that the new contour avoids all the singularities of $\R{z}$.
The central part of this integral (\eqref{eq:last} after the shift of the contour) gives an exponentially bounded term as per~\eqref{eq:Pt} with a constant which depends on $\sup_{\abs{b}\leq \beta}\norm{\R{-\epsilon, b}}_{\Bs} <\infty$.
This means that we merely need to estimate the norm of
\begin{equation}\label{eq:remains}
\lim_{k\to \infty} \int_{\beta}^{k} \exp\left( {- t \abs{b}^{-\Ctwelve}} \right) e^{ibt} \R{ i b - \abs{b}^{-\Ctwelve}   } \ db,
\end{equation}
and the similar integral from $-k$ to $-\beta$.
This will be postponed until the end of this section.
\end{proof}

Now we will need the following higher order version of Lemma~\ref{lem:Z}.
\begin{lem}\label{lem:ZZ}
Let $n\in \bN$. For all  $z \in \bC$ in the holomorphic domain of $\R{z}$ and $z\neq 0$, on $\cB(\operatorname{Dom}(Z^{n}), \Bs)$
\[
\R{z} = \frac{1}{z^{n}} \R{z}Z^{n}  +  \sum_{j=0}^{n-1} \frac{1}{z^{j+1}} Z^{j}.
\]
\end{lem}
\begin{proof}
The case $n=1$ is Lemma~\ref{lem:Z}. 
I.e., $\R{z} =  \frac{1}{z} \R{z}Z +  \frac{1}{z} \Id$.
Simply iterating this formula proves the result for all $n\in \bN$.
\end{proof}

For the following it is essential that Assumption~\ref{ass:rapid} is satisfied.
\begin{lem}
Let $n\in \bN$.
There exists $\Cthirteen>0$ such that
\[
\norm{\smash{  \R{  ib -\abs{b}^{-\Ctwelve}  } \mu }  }_{\Bs}
\leq
\Cthirteen \abs{b}^{\Celeven - n} \norm{ \mu }_{Z^{n}}
\]
for all $\mu \in \operatorname{Dom}(Z^{n})$, $b\in \bR$, $\abs{b} \geq \beta$.
\end{lem}
\begin{proof}
Using Lemma~\ref{lem:ZZ} we have
\[
\norm{\R{z}\mu } 
\leq
\frac{1}{\abs{z}^{n}} \norm{\R{z}}_{\Bs} \norm{ \mu }_{Z^{n}}  
+  \sum_{j=0}^{n-1} \frac{1}{\abs{z}^{j+1}} \norm{ \mu }_{Z^{j}}.
\]
We now substitute $z =  ib -\abs{b}^{-\Ctwelve}  $. 
Since $\norm{\R{z}}_{\Bs} \leq \Cten \abs{\Im(z)}^{\Celeven}    $ by Assumption~\ref{ass:rapid} there exists some $\Cthirteen > 0$ such that the lemma holds. 
\end{proof}

\begin{proof}[Proof of the second part of Theorem~\ref{thm:main2}]
We now use the above lemma to estimate the norm of the integral of~\eqref{eq:remains} and so complete the proof of  Theorem~\ref{thm:main2}.
\begin{multline*}
\lim_{k\to \infty}
 { \Big\|  \int_{\beta}^{k} \exp\left( {- t \abs{b}^{-\Ctwelve}} \right)  e^{ibt} \R{ i b - \abs{b}^{-\Ctwelve}   }\mu \ db \Big\|}_{\Bs}\\
 \leq
 \Cthirteen \left( 
 \int_{\beta}^{\infty} \exp\left( {- t \abs{b}^{-\Ctwelve}} \right)    \abs{b}^{\Celeven - q}   \ db
 \right)
 \norm{ \mu }_{Z^{q}}. 
\end{multline*}
This holds for any $q\in \bN$ but for our purposes we must choose $q$ large, in particular larger than $\Celeven$. 
Estimating the integral\footnote{Suppose that $a>0$, $n,k\in \bN$. 
Let $I(n) := \int_{0}^{a} e^{-tx}x^{n} \ dx$. 
Integrating by parts $I(n) \leq \frac{n}{t}I(n-1)$, and $I(0) \leq \frac{1}{t}$.
Consequently $I(n) \leq n!\ t^{-n}$. 
By a change of variables $x = y^{-k}$ the integral $I(n)$ is equal to $k \int_{a^{k}}^{\infty} \exp({-t y^{-k}}) y^{-(nk +k +1)}\ dy$.}
and choosing $q$ even larger depending also on the required rate of polynomial decay (denoted $p$ in the statement of the theorem) concludes the estimate.
\end{proof}




\begin{thebibliography}{10}

\bibitem{arendt2011vector}
W.~Arendt, C.~Batty, M.~Hieber, and F.~Neubrander.
\newblock {\em Vector-Valued Laplace Transforms and Cauchy Problems: Second
  Edition}.
\newblock Monographs in Mathematics. Birkh{\"a}user, 2011.

\bibitem{avila2005emt}
A.~Avila, S.~Gou{\"e}zel, and J.~Yoccoz.
\newblock Exponential mixing for the {T}eichm{\"u}ller flow.
\newblock {\em Publ. Math. Inst. Hautes \'Etudes Sci.}, 104(1):143--211, 2006.

\bibitem{Baladi:2011uq}
V.~Baladi and C.~Liverani.
\newblock Exponential decay of correlations for piecewise cone hyperbolic
  contact flows.
\newblock {\em Commun. Math. Phys.}, 314(3):689--773, 2012.

\bibitem{oli1202}
O.~Butterley.
\newblock Expanding semiflows on branched surfaces and one-parameter semigroups
  of operators.
\newblock {\em Nonlinearity}, 25(12):3487--3503, 2012.

\bibitem{oli1203}
O.~Butterley.
\newblock Area expanding $\mathcal{C}^{1+\alpha}$ suspension semiflows.
\newblock {\em Commun. Math. Phys.}, 325(2):803--820, 2014.

\bibitem{BL}
O.~Butterley and C.~Liverani.
\newblock Smooth {A}nosov flows: {C}orrelation spectra and stability.
\newblock {\em J. Mod. Dyn.}, 1(2):301--322, 2007.


\bibitem{BL2}
O.~Butterley and C.~Liverani.
\newblock  Robustly invariant sets in fiber contracting bundle flows.
\newblock {\em J. Mod. Dyn.}, 7(2):255--267, 2013.


\bibitem{Davies:2007qy}
E.~B. Davies.
\newblock {\em Linear Operators and Their Spectra}.
\newblock Number 106 in Cambridge studies in advanced mathematics. Cambridge
  University Press, 2007.

\bibitem{Demers:2011fk}
M.~F. Demers and H.-K. Zhang.
\newblock Spectral analysis of the transfer operator for the {L}orentz gas.
\newblock {\em J. Mod. Dyn.}, 5(4):665--709, 2011.

\bibitem{D}
D.~Dolgopyat.
\newblock On decay of correlations in {A}nosov flows.
\newblock {\em Ann. of Math. (2)}, 147:357--390, 1998.

\bibitem{dolgopyat1998prm}
D.~Dolgopyat.
\newblock Prevalence of rapid mixing in hyperbolic flows.
\newblock {\em Ergodic Theory Dynam. Systems}, 18(05):1097--1114, 1998.

\bibitem{dolgopyat2000prm}
D.~Dolgopyat.
\newblock Prevalence of rapid mixing--{II}: {T}opological prevalence.
\newblock {\em Ergodic Theory Dynam. Systems}, 20(04):1045--1059, 2000.

\bibitem{FMT}
M.~Field, I.~Melbourne, and A.~T\"or\"ok.
\newblock Stability of mixing and rapid mixing for hyperbolic flows.
\newblock {\em Ann. of Math. (2)}, 166(1):269--291, 2007.

\bibitem{Giulietti:fk}
P.~Giulietti, C.~Liverani, and M.~Pollicott.
\newblock Anosov flows and dynamical zeta functions.
\newblock {\em Ann. of Math. (2)}, 178(2):687--773, 2013.

\bibitem{He}
H.~Hennion.
\newblock Sur un th\'eor\`eme spectral et son application aux noyaux
  lipchitziens.
\newblock {\em Proc. Amer. Math. Soc.}, 118(2):627--634, 1993.

 \bibitem{Kato1966}
T.~Kato.
\newblock {Perturbation theory for linear operators}.
\newblock {\em Springer}, {1966}.

\bibitem{Keller:1989}
G.~Keller.
\newblock Markov extensions, zeta functions, and {F}redholm theory for
  piecewise invertible dynamical systems.
\newblock {\em Trans. Amer. Math. Soc.}, 314(2):433--497, 1989.

\bibitem{KL}
G.~Keller and C.~Liverani.
\newblock Stability of the spectrum for transfer operators.
\newblock {\em Ann. Sc. Norm. Super. Pisa Cl. Sci. (4)}, 28(1):141--152, 1999.

\bibitem{keller2005sgo}
G.~Keller and C.~Liverani.
\newblock A spectral gap for a one-dimensional lattice of coupled piecewise
  expanding interval maps.
\newblock In {\em Dynamics of Coupled Map Lattices and of Related Spatially
  Extended Systems}, volume 671 of {\em Lecture Notes in Physics}, pages
  115--151. Springer Berlin Heidelberg, 2005.

\bibitem{Li1}
C.~Liverani.
\newblock On contact {A}nosov flows.
\newblock {\em Ann. of Math. (2)}, 159:1275--1312, 2004.

\bibitem{melbourne2007rdc}
I.~Melbourne.
\newblock Rapid decay of correlations for nonuniformly hyperbolic flows.
\newblock {\em Trans. Amer. Math. Soc.}, 359(5):2421--2441, 2007.

\bibitem{MR2472164}
I.~Melbourne.
\newblock Decay of correlations for slowly mixing flows.
\newblock {\em Proc. Lond. Math. Soc. (3)}, 98(1):163--190, 2009.

\bibitem{Nussbaum}
R.~D. Nussbaum.
\newblock The radius of essential spectrum.
\newblock {\em Duke Math. J.}, 37:473--478, 1970.

\bibitem{Po}
M.~Pollicott.
\newblock On the rate of mixing of {A}xiom {A} flows.
\newblock {\em Invent. Math.}, 81(3):413--426, 1985.

\bibitem{Ts}
M.~Tsujii.
\newblock Decay of correlations in suspension semi-flows of angle multiplying
  maps.
\newblock {\em Ergod. Th. \& Dynam. Sys.}, 28(1):291--317, 2008.

\bibitem{tsujii2008qct}
M.~Tsujii.
\newblock Quasi-compactness of transfer operators for contact {A}nosov flows.
\newblock {\em Nonlinearity}, 23(7):1495--1545, 2010.

\bibitem{Tsujii:2011uq}
M.~Tsujii.
\newblock {C}ontact {A}nosov flows and the {F}ourier--{B}ros--{I}agolnitzer
  transform.
\newblock {\em Ergod. Th. \& Dynam. Sys.}, 32(6):2083--2118, 2012.

\end{thebibliography}
\end{document}